\documentclass[12pt]{amsart}

\usepackage{amsfonts,amssymb,amsmath,textcomp,amsthm}
\usepackage{color}
\usepackage{hyperref}
\usepackage{amstext} 
\usepackage{array}   
\newcolumntype{C}{>{$}c<{$}} 

\newtheorem{theorem}{Theorem}[section]
\theoremstyle{plain}

\newtheorem{corollary}[theorem]{Corollary}

\newtheorem{definition}[theorem]{Definition}
\newtheorem{example}[theorem]{Example}

\newtheorem{lemma}[theorem]{Lemma}

\newtheorem{proposition}[theorem]{Proposition}
\newtheorem{remark}[theorem]{Remark}

\numberwithin{equation}{section}

\begin{document}

\title{Half-automorphism group of a class of Bol loops}

\author[de Barros]{Dylene Agda Souza de Barros}
\author[dos Anjos]{Giliard Souza dos Anjos}
\address[de Barros]{Faculdade de Matem\'atica, Universidade Federal de Uberl\^{a}ndia, Av. Jo\~{a}o Naves de \'Avila, 2121, Campus Santa M\^onica, Uberl\^{a}ndia, MG, Brazil, CEP 38408-100}
\email[de Barros]{dylene@ufu.br}
\address[dos Anjos]{Departamento de Matem\'atica, Universidade Estadual Paulista (Unesp), Instituto de Bioci\^encias, Letras e Ci\^encias Exatas, Rua Crist\'ov\~ao Colombo, 2265, 15054-000, S\~ao Jos\'e do Rio Preto - SP, Brazil}
\email[dos Anjos]{giliard.anjos@unesp.br}

\begin{abstract}
A \emph{Bol loop} is a loop that satisfies the Bol identity $(xy.z)y=x(yz.y)$. If $L$ is a loop and  $f:L\to L$ is a bijection such that  $f(xy)\in\{f(x)f(y),f(y)f(x)\}$, for every $x$, $y\in L$, then $f$ is called a \emph{half-automorphism} of $L$. In this paper, we describe the half-automorphism group of a class of Bol loops of order $4m$.
\end{abstract}


\maketitle

\section{Introduction}

A \emph{loop} is a set $L$ with a binary operation $\cdot$ and a neutral element $1\in L$ such that for every $a$, $b\in L$ the equations $ax=b$ and $ya=b$ have unique solutions $x$, $y\in L$, respectively.

Let $(L,\ast)$ and $(L',\cdot)$ be loops. A bijection $f:L\to L'$ is a \emph{half-isomorphism} if $f(x\ast y)\in\{f(x)\cdot f(y),f(y)\cdot f(x)\}$, for every $x$, $y\in L$. A \emph{half-automorphism} is defined as expected. We say that a half-isomorphism (half-automorphism) is \emph{proper} if it is neither an isomorphism (automorphism) nor an anti-isomorphism (anti-automorphism). 

In the year 1957, W.R. Scott proved that there is no proper half-homomorphism between two groups. He also gave an example of a loop of order 8 that has a proper half-automorphism, so Scott's result can not be generalized to all loops. In the last decade, many facts about half-isomorphisms between nonassociative loops have been proved. Gagola and Giuliani extended Scott's result to Moufang loops of odd order \cite{GG}. Grishkov \emph{et al} showed that there is no proper half-automorphism of a finite automorphic Moufang loop \cite{GGRS}. Kinyon, Stuhl and Vojt\v{e}chovsk\'y generalized the previous results to a more general class of Moufang loops \cite{KSV}. In \cite{GA21}, Giuliani and dos Anjos showed that there is no proper half-isomorphism between automorphic loops of odd order. 

Investigations on loops that have proper half-automorphisms were also made. Gagola and Giuliani established conditions for the existence of proper half automorphisms for certain Moufang loops of even order, including Chein loops \cite{GG2}. A similar work was done by Giuliani, Plaumann and Sabinina for certain diassociative loops \cite{GPS}. In \cite{GA19}, Giuliani and dos Anjos described the half-automorphism group for a class of automorphic loops of even order, and the same was done for Chein loops in \cite{G21} . In this paper, we describe the half-automorphism group of the Bol loops of the form $L_M=\mathbb{Z}_2\times\mathbb{Z}_2\times M$, where $M$ is an abelian group of exponent greater than 2.

\section{Preliminaries}

Let $L$ be a loop. Denote by $Z(L)$, $N(L)$, $N_\lambda(L)$, $N_\mu(L)$ and $N_\rho(L)$ the \emph{center}, \emph{nucleus}, \emph{left nucleus}, \emph{middle nucleus} and \emph{right nucleus} of $L$, respectively. The \emph{commutant} of $L$ is defined by $C(L) = \{a\in L\,|\,\, ax = xa \,\, \forall \,\, x \in L\}$. The nuclei of $L$ are subgroups of $L$ and the center of $L$ is an abelian subgroup of $L$. The commutant of $L$ in general is not a subloop of $L$. 


A \emph{right Bol  loop}  is a loop that satisfies the right Bol identity 

\begin{equation}\label{rbol}
x((yz)y)=((xy)z)y,
\end{equation}
and a \emph{left Bol  loop}  is a loop that satisfies the left Bol identity 

\begin{equation}\label{lbol}
(x(yx))z = x(y(xz)).
\end{equation}

Right (left) Bol loops are power-associative, right (left) alternative, and have the right (left) inverse property. Other basic results about these loops can be found in \cite{P90,R66}.

A \emph{Moufang loop} is a loop that is both a right and a left Bol loop.

Let $(L,*)$ be a loop. The \emph{opposite loop} of $L$ is the loop $L^{op} = (L,\cdot)$ with the operation $\cdot$ defined by $x\cdot y = y*x$. If $L$ is a right Bol loop, then $L^{op}$ is a left Bol loop (and vice versa).

\begin{proposition}
\label{prop2} Let $L$ be a right Bol loop. If $L$ has an anti-automorphism, then $L$ is a Moufang loop.
\end{proposition}
\begin{proof}
 Let $\varphi$ be an anti-automorphism of $L$. Then $\varphi$ is an isomorphism from $L$ into $L^{op}$. Thus $L$ is also a left Bol loop, and hence $L$ is a Moufang loop.
\end{proof}
Consequently, if $\varphi$ is a half-automorphism of a right Bol loop $L$ that is not Moufang, then $\varphi$ is either a proper half-automorphism or an automorphism of $L$. The same can be shown for left Bol loops. In fact, every result about right Bol loops dualizes to left Bol loops. So, from now on, we will only work with right Bol loops, and we will call them simply Bol loops.

One can search for properties of half-isomorphisms of loops in \cite{GA20,GA21}. However, to make this text sort of self contained, some results follow:

\begin{proposition}
\label{prop0}(\cite[Proposition $2.2$]{GA20}) Let $Q$ and $Q'$ be loops and $f:Q \rightarrow Q'$ be a half-isomorphism. If $H$ is a subloop of $Q'$, then $f^{-1}(H)$ is a subloop of $Q$.
\end{proposition}

\begin{definition}
Let $L,L'$ be loops. A half-isomorphism $f:L \rightarrow L'$ is called \emph{special} if the inverse mapping $f^{-1}:L' \rightarrow L$ is also a half-isomorphism.
\end{definition}

\begin{proposition}
\label{prop11}(\cite[Corollary $2.7$]{GA20}) Every half-automorphism of a finite loop is special.
\end{proposition}

\begin{proposition}
\label{prop12}(\cite[Proposition $2.13$]{GA21}) Let $L$ and $L'$ be power-associative loops, and $f:L \to L'$ be a half-isomorphism. Then $f(x^n) =f(x)^n$, for all $x\in L$ and $n\in \mathbb{Z}$.
\end{proposition}

As a direct consequence of Proposition \ref{prop12}, we have that every half-isomorphism between power-associative loops preserves the order of the elements.

\section{A construction of a class of Bol loops}\label{bol}

We start this section pointing out a result from Foguel, Kinyon and Phillips \cite{FKP06} that establish when one can define a Bol loop over a right transversal of a group. Let $G$ be a group. A subset $B$ of $G$ is called a \emph{twisted subgroup} if $1,x^{-1},xyx\in B$, for all $x,y\in B$.

\begin{proposition}
\label{prop1}(\cite[Proposition $5.2$]{FKP06}) Let $G$ be a group, $H\leq G$, and $B\subset G$ a right transversal of  $H$ in $G$. If $B$ is a twisted subgroup of $G$, then $B$ with the operation

\begin{equation}\label{eqprop1}
x\cdot y = z, \textrm{ if } xy = hz, \textrm{ for some } h\in H,
\end{equation}
is a Bol loop. Conversely, if $H$ is core-free and $(B,\cdot)$ is a Bol loop, then $B$ is a twisted subgroup of $G$.
\end{proposition}

Let $M$ be an abelian group. Define $L_M = \mathbb{Z}_2\times \mathbb{Z}_2 \times M$ and consider the following operation on $L_M$:

\begin{equation}
\label{eq21}
(l,s,x)* (u,v,y) =  \left\{\begin{array}{rl}
(l,s,xy), & \textrm{if } u=v= 0, \\
(l+u,s+v,x^{-1}y), & \textrm{otherwise}.
\end{array}\right.
\end{equation}
In the following, we will show that $(L_M,*)$ is a Bol loop.

The \emph{generalized dihedral group of $M$} can be defined by $D(M) = M\cup Mr$, where $r\not \in M$, $r^2 = 1$ and $rxr = x^{-1}$, for every $x\in M$. Consider the direct product $G = \mathbb{Z}_2\times \mathbb{Z}_2 \times D(M)$. We have that $H = 0\times 0 \times \{1,r\}$ is a subgroup of $G$ of order $2$. Let 

\begin{center}
$B = \{(0,0,x),(l,s,rx)\,|\, x\in M,\,l,s\in \mathbb{Z}_2,(l,s)\not = (0,0)\}$.
\end{center}

The set $B$ is a right transversal of $H$ in $G$. Furthermore, $B$ contains $(0,0,1)$, the identity of $G$.

\begin{proposition}
\label{prop21} $B$ is a twisted subgroup of $G$.
\end{proposition}
\begin{proof}
Since $(l,s,rx)^{-1} = (l,s,rx)$ and $(0,0,x)^{-1} = (0,0,x^{-1})$, for every $l,s\in \mathbb{Z}_2$ and $x\in M$, we get that $B$ is closed under inverses. Let $x,y\in M$ and $l,s,u,v\in \mathbb{Z}_2$ be such that $(l,s)\not = (0,0)$ and $(u,v)\not = (0,0).$ Then

\begin{center}
$\begin{array}{lclcl}
(0,0,x)(0,0,y)(0,0,x) &=& (0,0,xyx),&&\\

(0,0,x)(l,s,ry)(0,0,x) &=& (l,s,x ry x) &=& (l,s,ry),\\

(l,s,rx)(0,0,y)(l,s,rx) &=& (2l,2s,rx y rx) &=& (0,0,y^{-1}),\\

(l,s,rx)(u,v,ry)(l,s,rx) &=& (2l+u,2s+v,rx ry rx) &=& (u,v,ry^{-1}x^2).
\end{array}$
\end{center}
Thus $XYX\in B$, for every $X,Y\in B$, and hence $B$ is a twisted subgroup of $G$.
\end{proof}

As a consequence of Propositions \ref{prop1} and \ref{prop21}, we have that $(B,\cdot)$ is a Bol loop, where $\cdot$ is the operation given by \eqref{eqprop1}. For $x,y\in M$ and $l,s,u,v\in \mathbb{Z}_2$ such that $(l,s)\not = (0,0)$, $(u,v)\not = (0,0)$ and $(l+u,s+v)\not = (0,0)$, we have:

\begin{center}
$\begin{array}{lcl}
(0,0,x)(0,0,y) &=& (0,0,1)(0,0,xy), \\
(0,0,x)(l,s,ry) &=& (0,0,1)(l,s,rx^{-1}y),\\
(l,s,rx)(0,0,y) &=& (0,0,1)(l,s,rxy),\\
(1,1,rx)(1,1,ry) &=& (0,0,1)(0,0,x^{-1}y),\\
(l,s,rx)(u,v,ry) &=& (0,0,r)(l+u,s+v,rx^{-1}y),
\end{array}$
\end{center}
and then the operation $\cdot$ is given by the following rules:

\begin{equation}
\label{eq22}
\begin{array}{lcl}
(0,0,x)\cdot (0,0,y) &=& (0,0,xy),\\
(0,0,x)\cdot (l,s,ry) &=& (l,s,rx^{-1}y),\\
(l,s,rx)\cdot (0,0,y) &=& (l,s,rxy),\\
(1,1,rx)\cdot (1,1,ry) &=& (0,0,x^{-1}y),\\
(l,s,rx)\cdot (u,v,ry) &=& (l+u,s+v,rx^{-1}y).
\end{array}
\end{equation}

Define $\psi: (L_n,*) \to (B,\cdot)$; $\psi((0,0,x))\mapsto (0,0,x)$ and $\psi((l,s,x))\mapsto (l,s,rx)$, where $(l,s)\not = (0,0)$. Comparing \eqref{eq21} and \eqref{eq22}, we get that $\psi$ is an isomorphism, and hence $(L_M,*)$ is a Bol loop.

\begin{proposition}
\label{prop22} If $M$ is an elementary abelian $2$-group, then $(L_M,*)$ is also an elementary abelian $2$-group. If $M$ is not an elementary abelian $2$-group, then $(L_M,*)$ is a nonassociative, noncommutative Bol loop.
\end{proposition}
\begin{proof}
If $M$ is an elementary abelian $2$-group, then $x^{-1} = x$, for every $x\in M$, and so  we can see that $(B,\cdot)$ is commutative and has exponent $2$ by \eqref{eq22}. Thus $(B,\cdot)$ is a commutative Moufang loop of exponent $2$, and hence it is an elementary abelian $2$-group. Now consider that $M$ is not an elementary abelian $2$-group. Then there exists $x\in M$ such that $x\not = x^{-1}$. Using \eqref{eq21}, we get:

\begin{center}
$\begin{array}{lclcl}
((0,1,x)*(1,0,x))*(0,1,x) &=& (1,1,0)*(0,1,x) &=& (1,0,x),\\
(0,1,x)*((1,0,x)*(0,1,x)) &=& (0,1,x)*(1,1,0) &=& (1,0,x^{-1}).
\end{array}$
\end{center}
Then $(L_M,*)$ is nonassociative and noncommutative.
\end{proof}

\begin{corollary}
\label{cor22}  If $M$ is an abelian group with exponent greater than 2, then $(L_M,*)$ is not a Moufang loop. In particular, $L_M$ has no anti-automorphisms.
\end{corollary}
\begin{proof}
In the proof of Proposition \ref{prop22}, we saw that $(L_M,*)$ does not satisfy the flexible identity, and then $(L_M,*)$ is not a Moufang loop. The rest of the claim follows from Proposition \ref{prop2}.
\end{proof}

Let $L$ be a power associative loop and let $x$ be an element of $L$. Then the order of $x$ will be denoted by $o(x)$. From \eqref{eq21}, we can easily obtain the order of the elements of $L_M$ and a straightfoward calculation gives us the nuclei, commutant and the center of $L_M$. 

\begin{proposition}
\label{prop23} Let $(u,v,x)$ be an element of $L_M$. Then
\begin{equation}
\nonumber
o((u,v,x)) =  \left\{\begin{array}{rl}
o(x), & \textrm{if } u=v= 0, \\
2, & \textrm{otherwise}.
\end{array}\right.
\end{equation}
\end{proposition}
\begin{proposition}
\label{prop24} For the Bol loop $L_M$, the following holds
\begin{center}
$N_ \lambda (L_M) = \{(u,v,x)\,|\, o(x) = 2\}, N_ \mu (L_M) =N_ \rho (L_M) = \{(0,0,x)\,|\, x\in M\}, \break $ and $  \mathcal{Z}(L_M) = N(L_M) = C(L_M) =  \{(0,0,x)\,|\, o(x) = 2\}.$
\end{center}
\end{proposition}

\section{Half-automorphisms of the Bol loop $L_M$}

In this section, suppose that $M$ is a finite abelian group with exponent greater than $2$. Let $L_M$ be the Bol loop constructed in section $\ref{bol}$. From now on, let us denote $L_M$ as $L_M=K\times M$, where $K=\{1,a,b,c\}$ is the Klein group and the multiplication in $L_M$ will be given by

\begin{center}
$\begin{array}{lclcl}
(1,x)&*&(1,y) &=& (1,xy)\\
(A,x)&*&(1,y) &=& (A,xy)\\
(1,x)&*&(B,y) &=& (B,x^{-1}y)\\
(A,x)&*&(B,y) &=& (AB,x^{-1}y),
\end{array}$
\end{center}
where $A,B\neq 1$.

Recall that $f:L_M\to L_M$ is a half-automorphism of $L_M$ if $f$ is a bijection of $L_M$ and $f(XY)\in\{f(X)f(Y),f(Y)f(X)\}$, for every $X$ and $Y$ in $L_M$. 

\begin{proposition}
\label{p31}Let $(1,M)=\{(1,x)\in L_M; x\in M\}$ and let $f$ be a half-automorphism of $L_M$. Then $f(1,M)=(1,M)$.
\end{proposition}
\begin{proof}
First, by Proposition \ref{prop12} it follows that, for every $X$ in $L_M$, $X$ and $f(X)$ have the same order. Let us denote the order of $X$ by $o(X)$. Note that if $(A,x)$ is an element of $L_M$ such that $A\neq1$,  then $o((A,x))=2$ and, for every $x\in M$, $o((1,x))=o(x)$. Since $exp(M)>2$, there exists an element $s$ in $M$ such that $o(s)>2$, then $f(1,s)=(1,s^{'})$ for some $s^{'}\in M$. Suppose, towards a contradiction, that there exists $(1,r)$ in $(1,M)$ such that $f(1,r)=(A,r^{'})$ for some $A\neq 1$ in $K$. Then $o(r)=o(1,r)=o(A,r^{'})=2$ and $(rs)^2=r^2s^2=s^2\neq1$, which implies $o(rs)>2$ and $f(1,rs)=(1,t)$ for some $t$. Thus 

\begin{center}
$f(1,t)=f(1,rs)=f((1,r)(1,s))\in\{f(1,r)f(1,s),f(1,s)f(1,r)\}=\{(A,r^{'})(1,s^{'}),(1,s^{'})(A,r^{'})\}=\{(A,r^{'}s^{'}),(A,r^{'}s^{'-1})\}$, 
\end{center}
which is a contradiction. Hence, $f(1,M)\subset(1,M)$ and, since $f$ is one-to-one and $(1,M)$ is finite,  $f(1,M)=(1,M)$.
\end{proof}

The set $(K,1)$ is a subgroup of $L_M$ isomorphic to $K$. By Propositions \ref{prop0} and \ref{prop11}, $f(K,1)$ is a subloop of $L_M$ of order $4$, and then it is a group isomorphic to $K$. Define 

\begin{center}
$\mathcal{H}_M = \{H\leq L_M\,\,;\,\, H\cong K, |H\cap (1,M)| = 1\}$.
\end{center}
\begin{proposition}
\label{pa32} Let $f$ be a half-automorphism of $L_M$. Then $f(K,1)\in \mathcal{H}_M$.
\end{proposition}

\begin{proof}
Let $H=f(K,1)$. We already saw that $H\cong K$. Since $f^{-1}(1,M)=(1,M)$, we have 

\begin{center}
$f^{-1}(H \cap (1,M))=f^{-1}(H) \cap f^{-1}(1,M)  = (K,1) \cap (1,M) = \{(1,1)\}$. 
\end{center}

Therefore, $|H \cap (1,M)|=1$.
\end{proof}

Now, let us describe the elements of $\mathcal{H}_M$. Let $H=\{(1,1),(A,x),(B,y),(C,z)\}$ be an element of $\mathcal{H}_M$, where $A,B,C\in K-\{1\}$ and $x,y,z\in M$. If $A=B$, then $(C,z)=(A,x)(B,y)=(A,x)(A,y)=(1,x^{-1}y)$ and $C=1$, which is a contradiction. Then $A\neq B$ and similarly, $A\neq C$ and $B\neq C$. So, $H=\{(1,1),(A,x),(B,y),(C,z)\}$, where $\{A,B,C\} = \{a,b,c\}$. Since   $(C,z) = (A,x)(B,y)=(B,y)(A,x)$, we get $z = x^{-1}y=xy^{-1}$, and then the order of $xy^{-1}$ divides $2$.  If $xy^{-1}=1$, then $x=y$ and $H=\{(1,1),(A,x),(B,x),(C,1)\}$. By $(A,x)(C,1)=(C,1)(A,x)$, we get $o(x)|2$. Then either $H = (K,1)$ or $H=\{(1,1),(A,x),(B,x),(C,1)\}$ with $o(x) = 2$. Suppose that the order of $xy^{-1}$ is equal to $2$. Then $x\not = y$ and $1 = xy^{-1}xy^{-1} = x^2y^{-2}$, and so $x^2 = y^2$. If $x = 1$, then $H=\{(1,1),(A,1),(B,y),(C,y)\}$, where $o(y) = 2$. If $y = 1$, then $H=\{(1,1),(A,x),(B,1),(C,x)\}$, where $o(x) = 2$. Consider $x,y\not = 1$. By $(A,x)(C,x^{-1}y)=(C,x^{-1}y)(A,x)$, we get $x^{-2}y =x^2y^{-1}$, and so $y^2 = x^4$. Since $x^2 = y^2$, we obtain $x^2 = x^4$ and then $x^2 = 1$. Hence, $H=\{(1,1),(A,x),(B,y),(C,xy)\}$, where $x,y\in M$ are distinct elements of order $2$. We have established the following result.

\begin{proposition}
\label{pa33} Let $H\in \mathcal{H}_M$. There are three possibilities:

(i) $H = (K,1)$,

(ii) $H = \{(1,1),(A,x),(B,x),(C,1)\}$, with $\{A,B,C\} = \{a,b,c\}$ and $o(x)=2$,

(iii) $H = \{(1,1),(A,x),(B,y),(C,xy)\}$, with $\{A,B,C\} = \{a,b,c\}$ and $x\not = y$ are elements of order equal to $2$.\\
In particular, if $M$ has odd order, then $\mathcal{H}_M = \{(K,1)\}$.
\end{proposition}

\begin{corollary}
\label{c31} If $H\in \mathcal{H}_M$, then $L_M = H(1,M) = (1,M)H$.
\end{corollary}
\begin{corollary}
\label{c31a} If $|M|$ is odd, then $(K,1)$ is the only subloop of order $4$ of $L_M$.
\end{corollary}
\begin{proof} Let $H$ be a subloop of $L_M$ of order $4$. Then $|H\cap (1,M)| = 1$. By Proposition ~\ref{prop23}, $o(X)\leq 2$, for every $X\in H$. Then $H\cong K$, and the result follows by Proposition \ref{pa33}.
\end{proof}

\begin{corollary}
\label{c32} Consider $M$ as a finite abelian group of even order and exponent greater than $2$ and write $M = C_{2i_1}\times C_{2i_2}\times... \times C_{2i_s} \times M_1$, where $|M_1|$ is an abelian group of odd order, $s\geq 1$ and $i_j\geq 1$. Then $|\mathcal{H}_M| = 4^s$.
\end{corollary}
\begin{proof} Each $C_{2i_j}$ has only one element of order $2$. Denote such element by $x_j$. Consider $n_1$, $n_2$ and $n_3$ as the numbers of elements in $\mathcal{H}_M$ of the types (i), (ii) and (iii) according to Proposition \ref{pa32}, respectively. Then $|\mathcal{H}_M| = n_1 + n_2 + n_3$ and $n_1 = 1$. 

We have that $n_2=3\cdot n_2'$, where $n_2'$ is the number of elements of order $2$ in $M$. Each such element has the form $x_{l_1}x_{l_2}...x_{l_r}$, where $r\leq s$ and $1\leq l_1<l_2<...<l_r\leq s$. Then $n_2' = 2^s - 1$ and we get $n_2 = 3\cdot (2^s - 1)$.

The number $n_3$ is equal to $6\cdot n_3'$, where $n_3'$ is the number of subsets of $M$ isomorphic to $K$. It is not difficult to see that $n_3' = \frac{1}{3}\cdot \frac{n_2'!}{2!(n_2'-2)!} = \frac{(2^s - 1)(2^s - 2)}{6}$. Then $n_3 = (2^s - 1)(2^s - 2)$ and we have

\begin{center}
$|\mathcal{H}_M| = n_1 + n_2 + n_3 = 1 + 3\cdot (2^s - 1) + (2^s - 1)(2^s - 2) = 4^s$.
\end{center}
\end{proof}

\begin{corollary}
\label{c33} Let $f$ be a half-automorphism of $L_M$. Then there exist a unique $f'\in Aut(K)$ and $x,y\in M$ such that $o(x),o(y)\leq 2$ and

\begin{center}
$f(A,1) = (f'(A),\alpha_{(x,y)}(A))$, for every $A\in K$,
\end{center}
where $\alpha_{(x,y)}(1) = 1$, $\alpha_{(x,y)}(a) = x$, $\alpha_{(x,y)}(b) = y$ and $\alpha_{(x,y)}(c) = xy$.
\end{corollary}
\begin{proof} The existence of $f'\in Aut(K)$ and $x,y\in M$ satisfying  the conditions of the claim is a consequence of Propositions \ref{pa32} and \ref{pa33} and the fact that $Aut(K)$ is the symmetric group $S_3$. The unicity is trivial.
\end{proof}
\begin{remark}
\label{obs1} In the conditions of Corollary \ref{c33}, the mapping $\alpha_{(x,y)}: K \to M$ is a homomorphism. Furthermore, from Proposition \ref{prop24} it follows that $(1,\alpha_{(x,y)}(A)) \in \mathcal{Z}(L_M)$, for every $A\in K$.
\end{remark}

\begin{proposition}
\label{autM} Let $f$ be a half-autmorphism of $L_M$. For every $x\in M$, consider $f''(x)\in M$ as $(1,f''(x))=f(1,x)$. Then $f'':M\to M$ is an automorphism of $M$.
\end{proposition}
\begin{proof}
First, observe that By Proposition \ref{p31}, $f''$ is well-defined. Now, for every $x$ and $y$ in $M$, we have 

\begin{center}
$(1,f''(xy))=f(1,xy)=f((1,x)(1,y))\in\{f(1,x)f(1,y), f(1,y)f(1,x)\}$.
\end{center}
Note that $\{f(1,x)f(1,y), f(1,y)f(1,x)\} = \{(1,f''(x)f''(y))\}$. Then $f''(xy)=f''(x)f''(y)$. 
\end{proof}

Every element $(A,x)$ in $L_M$ can be written as $(A,1)(1,x)$. By Corollary \ref{c33}, there are $f'\in Aut(K)$ and $u,v\in M$ such that $f(A,1) = (f'(A),\alpha_{(u,v)}(A))$, where $o(u),o(v)\leq 2$.
For $A\neq 1$, 

\begin{center}
$f(A,x)=f((A,1)(1,x))\in\{(f^{'}(A),1)(1,f''(x)),(1,f''(x))(f'(A),1)\}$,
\end{center}
and so

\begin{center}
$f(A,x) \in \{(f'(A),f''(x)\alpha_{(u,v)}(A)),(f'(A),f''(x^{-1})\alpha_{(u,v)}(A))\}$.
\end{center}
Then if $f'\in Aut(K)$, $f''\in Aut(M)$ and $u,v\in M$ are elements such that $o(u),o(v)\leq 2$, we define $F_{(f',f'',u,v)}^+,F_{(f',f'',u,v)}^-:L_M\to L_M$ by $$F_{(f',f'',u,v)}^+(A,x)=(f'(A),f''(x)\alpha_{(u,v)}(A))\quad\mbox{and}$$ $$F_{(f',f'',u,v)}^-(A,x)=\left\{\begin{array}{rc}
(f'(A),f''(x)\alpha_{(u,v)}(A)),&\mbox{if}\quad A= 1,\\
(f'(A),f''(x^{-1})\alpha_{(u,v)}(A)), &\mbox{otherwise}.
\end{array}\right.
$$
For making the notation easier, we will write $f_{(u,v)}^+$ and $f_{(u,v)}^-$ instead of $F_{(f',f'',u,v)}^+$ and $F_{(f',f'',u,v)}^-$, respectively. By Corollary \ref{c31}, $f_{(u,v)}^+$ and $f_{(u,v)}^-$ are bijections. Furthermore, for every $A\in K$ and $x\in M$,
\begin{equation}
\label{e31}
f_{(u,v)}^+(A,x)\!=\! (1,\alpha_{(u,v)}(A))\!*\!f_{(1,1)}^+(A,x) \textrm{ and } f_{(u,v)}^-(A,x)\!=\! (1,\alpha_{(u,v)}(A))\!*\!f_{(1,1)}^-(A,x).
\end{equation}
\begin{proposition}
\label{p34} In the conditions above, $f_{(u,v)}^+$ is an automorphism and $f_{(u,v)}^-$ is a proper half-automorphism of $L_M$.
\end{proposition}
\begin{proof}
Since $\alpha_{(u,v)}$ is a homomorphism and $(1,\alpha_{(u,v)}(A))\in \mathcal{Z}(L_M)$, for every $A\in K$ (Remark \ref{obs1}), we only have to prove that $f_{(1,1)}^+$ is an automorphism and $f_{(1,1)}^-$ is a proper half-automorphism by \eqref{e31}. Denote $f_{(1,1)}^+$ and $f_{(1,1)}^-$ by $f^+$ and $f^-$, respectively. First, let us prove that $f^+$ is an automorphism. Let $(A,x),(B,y)\in L_M$. When $B=1$ we have  

\begin{center}
$f^+((A,x)(1,y))=f^+(A,xy)=(f'(A),f''(xy))=(f'(A),f''(x))(1,f''(y))=f^+(A,x)f^+(1,y)$. 
\end{center}

\noindent{}If $B\neq 1$, then 

\begin{center}
$f^+((A,x)(B,y))= f^+(AB,x^{-1}y)= (f'(AB),f''(x^{-1}y))=(f'(A)f'(B),f''(x)^{-1}f''(y))=f^+(A,x)f^+(B,y)$.
\end{center} 
 
\noindent{}Now, let us prove that $f^-$ is a proper half-automorphism. Clearly, $f^-((1,x)(1,y))=f^-(1,x)f^-(1,y)$ for every $x$, $y\in M$. A straightfoward calculation shows us that, for $B\neq 1$, we have $f^-((1,x)(B,y))=f^-(B,y)f^-(1,x)$. Also, for $A\neq 1$, we get $f^-((A,x)(1,y))=f^-(1,y)f^-(A,x)$ and $f^-((A,x)(A,y))=f^-(A,y)f^-(A,x)$. Finally, for $A\neq B$, and both not equal to $1$, we have
 
\begin{center}
$f^-((A,x)(B,y))=f^-(AB,x^{-1}y)=(f'(AB),f''(xy^{-1})) =f^-(A,x)f^-(B,y)$. 
\end{center}

\noindent{}On other hand, we can note that $f^-(B,y)f^-(A,x)=(f'(AB),f''(y)f''(x^{-1}))$
and, since there exists $w\in M$ such that $w\not = w^{-1}$, we conclude that $f^-$ is a proper half-automorphism of $L_M$.
\end{proof}

\begin{proposition}
\label{auto}
Let $g$ be an automorphism of $L_M$. Then $g=g^{+}_{(u,v)}$.
\end{proposition}
\begin{proof}
First, observe that $g^{+}_{(u,v)}(1,y)=g^{-}_{(u,v)}(1,y)$ for every $y\in M$ and, if $o(y)\leq2$,  $g^{+}_{(u,v)}(A,y)=g^{-}_{(u,v)}(A,y)$ for every $A\in K$. Then suppose, towards a contradiction, that $g(A,y)=g^{-}_{(u,v)}(A,y)$ for some $A\in K-\{1\}$ and $y\in M$ with $o(y)>2$. For every $x\in M$, $g((A,y)(A,x))=g(A,y)g(A,x)$ and $g((A,y)(A,x))=g(1,y^{-1}x)=(1,g''(y^{-1}x))$. If $g(A,x)=g^{+}_{(u,v)}(A,x)$, then $g(A,y)g(A,x)=(1,g''(y)g''(x))$ which implies $g''(y)=g''(y^{-1})$, which is a contradiction since $o(y)>2$. On other hand, if $g(A,x)=g^{-}_{(u,v)}(A,x)$, then $g(A,y)g(A,x)=(1,g''(y)g''(x^{-1}))$, and so $o(y^{-1}x)\leq 2$ for every $x\in M$, which is a contradiction because $M$ has exponent greater than 2. Therefore, $g=g^{+}_{(u,v)}$.
\end{proof}

\begin{proposition}
\label{p36}
Let $g$ be a proper half-automorphism of $L_M$. Then $g=g^{-}_{(u,v)}$.
\end{proposition}
\begin{proof}
For every $(A,x)\in L_M$, $g(A,x)\in\{g^{+}_{(u,v)}(A,x),g^{-}_{(u,v)}(A,x)\}$ and, by Proposition ~\ref{auto}, $g\neq g^{+}_{(u,v)}$. Also, $g^{+}_{(u,v)}(A,x)=g^{-}_{(u,v)}(A,x)$ whenever either $A= 1$ or $o(x)\leq 2$. Thus suppose, towards a contradiction, that $g(A,x)= g^{+}_{(u,v)}(A,x)$ for some $A\in K-\{1\}$ and $x\in M$ with $o(x)>2$. Since $g$ is not an automorphim of $L_M$, there is $(\epsilon,\theta)\in L_M$ such that $g(\epsilon,\theta)=g^{-}_{(u,v)}(\epsilon,\theta)\neq g^{+}_{(u,v)}(\epsilon,\theta)$ and then $\epsilon\neq 1$ and $o(\theta)>2$. First, suppose $\epsilon=A$ and set $\theta=z$, that is, $g(A,z)=(g'(A),g''(z^{-1})\alpha(u,v)(A))$. Since 

\begin{center}
$g((A,x)(A,z))=g(1,x^{-1}z)=(1,g''(x^{-1}z))\in \{g(A,x)g(A,z),g(A,z)g(A,x)\}$,
\end{center}

\noindent{}$g(A,x)g(A,z)=(1,g''(x^{-1}z^{-1}))$ and $g(A,z)g(A,x)=(1,g''(zx))$, we get either $x$ or $z$ with order less then or equal to $2$, which is a contradiction. Now, suppose $\epsilon=AB$, with $B\in K-\{1,A\}$ and set $\theta=z$. We have 

\begin{center}
$(A,x)(AB,z)=(B,x^{-1}z)\in\{g^{+}_{(u,v)}(B,x^{-1}z),g^{-}_{(u,v)}(B,x^{-1}z)\}$ 
\end{center}

\noindent{}and $g((A,x)(AB,z))\in\{g(A,x)g(AB,z),g(AB,z)g(A,x)\}$. First, if $g(B,x^{-1}z)=g^{+}_{(u,v)}(B,x^{-1}z)$, then we have either $g(A,x)g(AB,z)=(g'(B),g''(x^{-1}z^{-1})\alpha(u,v)(B))$ and $z=z^{-1}$ or 
$g(AB,z)g(A,x)\!\!=(g'(B),g''(zx)\alpha(u,v)(B))$ and $x=x^{-1}$ and, in both  cases, we get a contradiction. Finally, if $g(B,x^{-1}z)=g^{-}_{(u,v)}(B,x^{-1}z)$, in the same way as before, we get either $x=x^{-1}$ or $z=z^{-1}$ and so, a contradiction. Hence, $g=g^{-}_{(u,v)}$ as desired.
\end{proof}
\begin{theorem}
\label{t31} Let $L_M$ be the Bol loop constructed in the Section \ref{bol} and let $Half(L_M)$ be the group of half-automorphisms of $L_M$. Then $Half(L_M)$ is the set

\begin{center}
$\{F_{(f',f'',u,v)}^+,F_{(f',f'',u,v)}^-\,|\,f'\in Aut(K), f''\in Aut(M), u,v\in M,o(u),o(v)\leq 2\}$.
\end{center}
Furthermore:

(a) If $|M|$ is odd, then $|Half(L_M)| = 2.|Aut(K)|.|Aut(M)|$.

(b) If $|M|$ is even, then $|Half(L_M)| = 2^{2s+1}.|Aut(K)|.|Aut(M)|$, where the abelian group $M$ is the direct product $C_{2i_1}\times C_{2i_2}\times... \times C_{2i_s} \times M_1$, $|M_1|$ has odd order, $s\geq 1$ and $i_j\geq 1$, for all $j$, and $C_{n}$ denotes the cyclic group of order $n$.
\end{theorem}
\begin{proof} By Corollary \ref{cor22} and Propositions \ref{auto} and \ref{p36}, we have that $Half(L_M)$ is the set of all mappings $F_{(f',f'',u,v)}^+$ and $F_{(f',f'',u,v)}^-$. If $|M|$ is odd, then $M$ has no elements of order $2$ and we get (a). Now consider that $|M|$ is even and $M = C_{2i_1}\times C_{2i_2}\times... \times C_{2i_s} \times M_1$. In the proof of Corollary \ref{c32}, we saw that $M$ has $2^s - 1$ elements of order $2$. Then $|Half(L_M)| = 2.(2^s)^2.|Aut(K)|.|Aut(M)| =  2^{2s+1}.|Aut(K)|.|Aut(M)|$.
\end{proof}

\section{Half-automorphisms group of $L_M$}\label{haut}

In this section we present the structure of the half-automorphisms group of the loop $L_M$. Recall that if $f$ is a half-automorphism of $L_M$, then by Corollary 
\ref{c33} and Proposition \ref{autM} the mappings $f'$ and $f''$ automorphisms of $K$ and $M$, respectively.
\begin{lemma}
\label{l41} If $f,g$ are half-automorphisms of $L_M$, then
\begin{itemize}
    \item[(a)] $(gf)' = g'f'$ and
    \item[(b)] $(gf)'' = g''f''$.
\end{itemize}
\end{lemma}
\begin{proof} (a) By Corollary \ref{c33}, $f(A,1) = (f'(A),\alpha_{(x,y)}(A))$, $g(A,1) = (g'(A),\alpha_{(x',y')}(A))$ and $(gf)(A,1) = ((gf)'(A),\alpha_{(x'',y'')}(A))$, for every $A\in K$. Then for $A\in K$, we have

\begin{center}
$((gf)'(A),\alpha_{(x'',y'')}(A)) = (gf)(A,1) = g((f'(A),\alpha_{(x,y)}(A)))$.
\end{center}
Since $g$ is a half-automorphism and $g(1,\alpha_{(x,y)}(A)) = (1,g''(\alpha_{(x,y)}(A)))$, there exists $\epsilon\in \{-1,1\}$ such that $g((f'(A),\alpha_{(x,y)}(A))) = (g'f'(A),g''(\alpha_{(x,y)}(A)^\epsilon)\alpha_{(x',y')}(A))$. Then $(gf)'(A) = g'f'(A)$.\\
(b) For $x\in M$, we have that 

\begin{center}
$(1,(gf)''(x)) = (gf)(1,x) = g(1,f''(x)) = (1,g''f''(x))$,
\end{center}
and then $(gf)'' = g''f''$.
\end{proof}

Setting $K$ as $\{1,a,b,c\}$, the elements of $Aut(K)$ are the permutations $I$, $(a\, b)$, $(a\, c)$, $(b\, c)$, $(a\, b\, c)$ and $ (a\, c \, b)$, where $I$ is the identity mapping. Let $u,v$ be elements of $M$ such that  $o(u),o(v)\leq 2$. It is easy to see that, for all $A\in K$:

\begin{eqnarray}
\label{ess5}\nonumber 
\alpha_{(u,v)}(I(A)) = \alpha_{(u,v)}(A),\quad  \alpha_{(u,v)}((a\, b)(A)) = \alpha_{(v,u)}(A),\\
\alpha_{(u,v)}((a\, c)(A)) = \alpha_{(uv,v)}(A),\quad \alpha_{(u,v)}((b\, c)(A)) = \alpha_{(u,uv)}(A),\\ 
 \alpha_{(u,v)}((a\, b\, c)(A)) = \alpha_{(v,uv)}(A),\quad \alpha_{(u,v)}((a\, c \, b)(A)) = \alpha_{(uv,u)}(A).
\nonumber
\end{eqnarray}
Then for $f'\in Aut(K)$ and $A\in K$, we have
\begin{equation}
\label{eq41}
\alpha_{(u,v)}(f'(A)) = \alpha_{(\overline{f'}_1(u,v),\overline{f'}_2(u,v))}(A),
\end{equation}
for some mappings $\overline{f'}_1,\overline{f'}_2: M\times M \to M$. When $u\not = v$, we get $\overline{f'}_1(u,v) = \overline{f'}(u)$ and $\overline{f'}_2(u,v) = \overline{f'}(v)$, where:
\begin{equation}
\label{ess4}
\begin{array}{lcl}
& &\overline{I}(u) = u, \overline{I}(v) = v,\overline{(a\, b)} (u) = v, \overline{(a\, b)} (v) = u,\\\vspace{0.1cm} & &\overline{(a\, c)} (u) = uv, \overline{(a\, c)} (v) = v, \overline{(b\, c)} (u) = u, \overline{(b\, c)} (v) = uv,\\\vspace{0.1cm} & &\overline{(a\, b\, c)} (u) = v, \overline{(a\, b\, c)} (v) = uv,  \overline{(a\, c \, b)} (u) = uv, \overline{(a\, c \, b)} (v) = u.
\end{array}
\end{equation}

It is known that if $g''\in Aut(M)$, then $o(g''(u)) = o(u)$ and $o(g''(v)) = o(v)$, and, therefore, $\alpha_{(g''(u),g''(v))}$ is well-defined. Moreover, a straightfoward calculation shows us the following
\begin{lemma}
\label{l42} Let $g''\in Aut(M)$ and $u,v,u',v'\in M$ be such that their orders are smaller or equal to $2$. Then for all $A\in K$:
\begin{itemize}
    \item[(a)] $g''(\alpha_{(u,v)}(A)) = \alpha_{(g''(u),g''(v))}(A).$
    \item[(b)] $\alpha_{(u,v)}(A)\alpha_{(u',v')}(A) = \alpha_{(uu',vv')}(A).$
\end{itemize}
\end{lemma}

Theorem \ref{t31} gives us a label of every element of $Half(L_M)$. In the following, we describe the behavior of these elements in such label.
\begin{proposition}
\label{p41} Let  $ F^+_{(f',f'',u,v)}, F^-_{(f',f'',u,v)}, F^+_{(g',g'',u',v')},F^-_{(g',g'',u',v')}$ be half-automorphisms of $L_M$, where $f', g'$ are automorphisms of  the Klein group $K$, $f'', g''$ are automorphisms of the abelian group $M$ and $u, v$ are elements of $M$ of order at most $2$. Then
\begin{itemize}
    \item[(a)] $F^+_{(g',g'',u',v')} F^+_{(f',f'',u,v)} = F^+_{(g'f',g''f'',g''(u)\overline{f'}_1(u',v'),g''(v)\overline{f'}_2(u',v'))}$,
    \item[(b)] $F^+_{(g',g'',u',v')} F^-_{(f',f'',u,v)} = F^-_{(g'f',g''f'',g''(u)\overline{f'}_1(u',v'),g''(v)\overline{f'}_2(u',v'))}$,
    \item[(c)] $F^-_{(g',g'',u',v')} F^+_{(f',f'',u,v)} = F^-_{(g'f',g''f'',g''(u)\overline{f'}_1(u',v'),g''(v)\overline{f'}_2(u',v'))}$ and
    \item[(d)] $F^-_{(g',g'',u',v')} F^-_{(f',f'',u,v)} = F^+_{(g'f',g''f'',g''(u)\overline{f'}_1(u',v'),g''(v)\overline{f'}_2(u',v'))}$. 
\end{itemize}
\end{proposition}
\begin{proof} We will just prove item (b). The others can be, similarly, proved. Denote $F^+_{(g',g'',u',v')}$ and $ F^-_{(f',f'',u,v)}$ by $g$ and $f$, respectively. Then for every $A$, $B$ in $K$ and $x$, $y$ in M, we have $g(B,y)=(g'(B),g''(y)\alpha_{(u',v')}(B))$ and 

\begin{center}
$f(A,x)=\left\{\begin{array}{rc}
(1,f''(x)),&\mbox{if}\quad A= 1,\\
(f'(A),f''(x^{-1})\alpha_{(u,v)}(A)), &\mbox{otherwise}.
\end{array}\right.$
\end{center}
If $A = 1$, then 

\begin{center}
$g(f(A,x)) = g((1,f''(x))) = (1,g''f''(x))$,
\end{center}
and if $A\not = 1$, 

\begin{flushleft}
$\begin{array}{lcl}
g(f(A,x)) &=& g((f'(A),f''(x^{-1})\alpha_{(u,v)}(A))) \\
 &=&  (g'(f'(A)),g''(f''(x^{-1})\alpha_{(u,v)}(A))\alpha_{(u',v')}(f'(A)))\\
 &=&  (g'(f'(A)),g''(f''(x^{-1}))g''(\alpha_{(u,v)}(A))\alpha_{(\overline{f'}_1(u',v'),\overline{f'}_2(u',v'))}(A))\\
 &=& (g'f'(A),g''f''(x^{-1})\alpha_{(g''(u)\overline{f'}_1(u',v'),g''(v)\overline{f'}_2(u',v'))}(A)).
\end{array}$
\end{flushleft}
\end{proof}
\begin{corollary}
\label{c41} $Half(L_M) \cong C_2\times Aut(L_M)$.
\end{corollary}
\begin{proof} It is a consequence of Propositions \ref{p34}, \ref{auto} and \ref{p41}.
\end{proof}

In order to reach a full description of the group $Half(L_M)$, we aim to describe the automorphisms group $Aut(L_M)$.
\begin{proposition}
\label{p42} Let $\mathcal{A} = \{F_{(f',f'',1,1)}^+\,|\,f'\in Aut(K), f''\in Aut(M)\}$. Then
 \begin{center}
$\mathcal{A} \cong Aut(K) \times Aut(M)$.
\end{center}
\end{proposition}

\begin{proof} Let $F_{(f',f'',1,1)}^+,F_{(g',g'',1,1)}^+ \in \mathcal{A}$ and denote $f = F_{(f',f'',1,1)}^+$ and $g = F_{(g',g'',1,1)}^+$. Then $f(A,x) = (f'(A),f''(x))$ and $g(A,x) = (g'(A),g''(x))$, for every $A\in K$ and $x\in M$. Also

\begin{equation}
\label{eqp42}
gf(A,x) = (g'f'(A),g''f''(x)) = ((gf)'(A),(gf)''(x)).
\end{equation}
Define the mapping $\Psi: \mathcal{A} \to Aut(K) \times Aut(M)$ by $\Psi(F_{(\phi,\varphi,1,1)}^+) = (\phi,\varphi)$. Clearly, $\Psi$ is a bijection and, furthermore,

\begin{center}
$\Psi(gf) = ((gf)',(gf)'') = (g'f',g''f'') = (g',g'')(f',f'') = \psi(g)\psi(f)$.
\end{center}
Therefore, $\Psi$ is an isomorphism from $\mathcal{A}$ to $Aut(K) \times Aut(M)$.
\end{proof}

The following result allows us to describe the group $Half(L_M)$ in terms of the automorphism group of $M$, for a special case of $M$.

\begin{theorem}
\label{c42} Let $K$ be the Klein group, let $M$ be an abelian group of odd order and let $L_M=K\times M$ be the Bol loop constructed in Section $\ref{bol}$. Then
\begin{center}
$Aut(L_M)\cong S_3 \times Aut(M)$ and $Half(L_M) \cong C_2\times S_3 \times Aut(M)$.
\end{center}
\end{theorem}

\begin{proof} First observe that $M$ has no elements of order $2$. By Proposition ~\ref{auto} and Theorem~\ref{t31}, $Aut(L_M) = \{F_{(f',f'',1,1)}^+\,|\,f'\in Aut(K), f''\in Aut(M)\}$. Since $Aut(K) \cong S_3$, the result follows from Corollary \ref{c41}, Proposition \ref{p42}. 
\end{proof}

From now on we consider $|M|$ even. Then $M = C_{2i_1}\times C_{2i_2}\times... \times C_{2i_s} \times M_1$, where $M_1$ is an abelian group of odd order, $s\geq 1$ and $i_j\geq 1$, for all $j$. Set $H = \{x\in M\,|\, o(x)\leq 2\}$. It is known that $H$ is a subgroup of $M$ of exponent $2$ and, by the proof of Corollary ~\ref{c32}, we have that $|H| = 2^s$. Denote by $I_K$ and $I_M$ the identity mappings of $K$ and $M$, respectively. Define

\begin{center}
$\mathcal{B} = \{F_{(I_K,I_M,x,y)}^+\,|\,x,y\in H\}$.
\end{center}
By Proposition \ref{p41} (a), $\mathcal{B}$ is closed under compositions, and, since $M$ is finite, $\mathcal{B}$ is a subgroup of $Aut(L_M)$. Considering $\mathcal{A}$ as in Proposition \ref{p42}, we have that $|\mathcal{A} \cap \mathcal{B}| = 1$, and so 

\begin{equation}
\label{eq57} |\mathcal{A}\, \mathcal{B}| = |\mathcal{A}|.|\mathcal{B}| = 2^{2s}.|Aut(K)|.|Aut(M)|
\end{equation}

\begin{proposition}
\label{p43} Let $K$ be the Klein group, let $M$ be an abelian group of even order and let $L_M=K\times M$ be the Bol loop constructed in Section \ref{bol}. $Aut(L_M) = \mathcal{A}\, \mathcal{B} =  \mathcal{B}\, \mathcal{A}$.
\end{proposition}
\begin{proof} By Theorem \ref{t31} and Propositon \ref{auto}, $|Aut(L_M)| = 2^{2s}.|Aut(K)|.|Aut(M)|$. Then the result follows from Equation \ref{eq57}.
\end{proof}

\begin{proposition}
\label{p44} Using the notation above, 
\begin{itemize}
    \item[(a)] $\mathcal{B} \cong H \times H \cong C_2^{2s}$,
    \item[(b)] $\mathcal{B} \triangleleft Aut(L_M)$ and
    \item[(c)] $Aut(L_M)/\mathcal{B} \cong \mathcal{A}$.
\end{itemize}
\end{proposition}
\begin{proof} (a) We only have to prove that $\mathcal{B}$ has exponent $2$. Let $f = F_{(I_K,I_M,u,v)}^+ \in \mathcal{B}$. Since $f(A,x) = (A,x\alpha_{(u,v)}(A))$, we have that

\begin{center}
$f^2(A,x) = f(A,x\alpha_{(u,v)}(A)) = (A,x\alpha_{(u,v)}(A)\alpha_{(u,v)}(A)) = (A,x)$.
\end{center}
Hence, $f^2$ is the identity mapping of $L_M$ and we have that $\mathcal{B}$ has exponent $2$.\\
\noindent{}(b) and (c) Define the mapping $\psi: Aut(L_M) \to \mathcal{A}$, $\psi(F^+_{(f',f'',u,v)}) = F^+_{(f',f'',1,1)}$. It is clear that $\psi$ is surjective. Let $F^+_{(f',f'',u,v)}, F^+_{(g',g'',u',v')}\in Aut(L_M)$. By Proposition \ref{p41}, we have that  $\psi(F^+_{(g',g'',u',v')}F^+_{(f',f'',u,v)})= F^+_{(g'f',g''f'',1,1)}$, and using \eqref{eqp42} we get

\begin{center}
$\psi(F^+_{(g',g'',u',v')}F^+_{(f',f'',u,v)}) = F^+_{(g',g'',1,1)}F^+_{(f',f'',1,1)} = \psi(F^+_{(g',g'',u',v')})\psi(F^+_{(f',f'',u,v)})$.
\end{center}
Then $\psi$ is a homomorphism. Now note that 

\begin{center}
$Ker(\psi) = \{F_{(I_K,I_M,u,v)}^+\,|\,u,v\in H\} = \mathcal{B}$.
\end{center}
Hence, $\mathcal{B} \triangleleft Aut(L_M)$ and $Aut(L_M)/\mathcal{B} \cong \mathcal{A}$.
\end{proof}
Define $\sigma: \mathcal{A} \to \mathcal{B}$, where, for $\alpha \in \mathcal{A}$, $\sigma(\alpha) = \sigma_\alpha$ and $\sigma_\alpha(\beta) = \alpha \beta \alpha^{-1}$, for all $\beta\in \mathcal{B}$. Since $\mathcal{B} \triangleleft Aut(L_M)$, then $\sigma$ is a homomorphism and we can contruct the inner semidirect product $\mathcal{A} \stackrel{\sigma}{\ltimes} \mathcal{B}$, where the operation is given by

\begin{center}
$(\alpha,\beta)\cdot (\alpha',\beta') = (\alpha\alpha',\beta\alpha \beta' \alpha^{-1})$, ($\alpha,\alpha'\in \mathcal{A}$, $\beta,\beta'\in \mathcal{B}$).
\end{center}
Define $\psi: Aut(L_M) \to \mathcal{A} \stackrel{\sigma}{\ltimes} \mathcal{B}$ by $\psi(\beta \alpha) = (\alpha,\beta)$. Then $\psi$ is a bijection. Furthermore,

\begin{center}
$\psi((\beta \alpha)(\beta' \alpha')) = \psi(\beta (\alpha \beta'\alpha^{-1}) \alpha \alpha') = (\alpha \alpha',\beta (\alpha \beta'\alpha^{-1})) = (\alpha,\beta)\cdot (\alpha',\beta')$.
\end{center}
Thus  $Aut(L_M) \cong \mathcal{A} \stackrel{\sigma}{\ltimes} \mathcal{B}$. Hence, we established the following result.

\begin{theorem}
\label{t41} Let $M$ be a finite abelian group such that its exponent is greater than $2$. Write $M = C_{2i_1}\times C_{2i_2}\times... \times C_{2i_s} \times M_1$, where $M_1$ is an abelian group of odd order, $s\geq 1$ and $i_j\geq 1$, for all $j$. Then
\begin{center}
$Aut(L_M) \cong \mathcal{A} \stackrel{\sigma}{\ltimes} \mathcal{B}$ and $Half(L_M) \cong C_2\times (\mathcal{A} \stackrel{\sigma}{\ltimes} \mathcal{B})$,
\end{center}
where $\mathcal{A} \cong S_3 \times Aut(M)$ and $\mathcal{B} \cong C_2^{2s}$.
\end{theorem}

We finish this paper with two examples.
\begin{example}
\label{ex1}
Let $M = C_3$, the cyclic group of order $3$. Then $L_M$ is a nonassociative Bol loop of order $12$, which is recognized by the command ``RightBolLoop$(12,3)$'' in the library of  loops of the LOOPS package \cite{NV1} for GAP \cite{gap}. The Cayley table of this loop is presented below.

\begin{center}
\begin{tabular}{C|CCCCCCCCCCCC}
*&1&2&3&4&5&6&7&8&9&10&11&12\\
\hline
1&1&2&3&4&5&6&7&8&9&10&11&12\\
2&2&1&4&3&6&5&8&10&11&9&12&7\\
3&3&5&6&2&4&1&10&9&12&11&7&8\\
4&4&6&5&1&3&2&9&11&7&12&8&10\\
5&5&3&2&6&1&4&12&7&10&8&9&11\\
6&6&4&1&5&2&3&11&12&8&7&10&9\\
7&7&9&11&8&12&10&1&5&4&6&3&2\\
8&8&10&12&7&11&9&2&1&6&5&4&3\\
9&9&7&8&11&10&12&4&3&1&2&5&6\\
10&10&8&7&12&9&11&3&2&5&1&6&4\\
11&11&12&10&9&8&7&6&4&2&3&1&5\\
12&12&11&9&10&7&8&5&6&3&4&2&1\\
\end{tabular}
\end{center}

Since $Aut(M) = C_2$, we have that

\begin{center}
$Aut(L_M) \cong C_2 \times S_3$ and $Half(L_M) \cong C_2^2 \times S_3$
\end{center}
by Theorem \ref{c42}. Then $L_M$ has $24$ half-automorphisms, from which $12$ are proper. By using GAP computing, we can obtain expressions for these mappings in terms of permutations. The automorphisms of $L_M$ are the permutations:
\begin{center}
$\begin{array}{ccc}
I_d,&( 7, 9)( 8,11)(10,12),&( 2, 8,11)( 4,12,10)( 5, 9, 7),\\
( 2, 9,11, 5, 8, 7)( 3, 6)( 4,12,10),&( 2,11)( 4,10)( 5, 7),&( 2,11, 8)( 4,10,12)( 5, 7, 9),\\
( 2, 5)( 3, 6)( 7, 8)( 9,11)(10,12),&( 2, 5)( 3, 6)( 7,11)( 8, 9),&( 2, 7)( 3, 6)( 4,10)( 5,11)( 8, 9),\\
( 2, 7, 8, 5,11, 9)( 3, 6)( 4,10,12),&( 2, 8)( 4,12)( 5, 9),&( 2, 9)( 3, 6)( 4,12)( 5, 8)( 7,11)
\end{array}$
\end{center}
and the proper half-automorphisms of $L_M$ are the permutations:
\begin{center}
$\begin{array}{ccc}
( 2, 5)( 7, 8)( 9,11)(10,12),&( 2, 5)( 7,11)( 8, 9),&( 2, 7)( 4,10)( 5,11)( 8, 9),\\
( 2, 7, 8, 5,11, 9)( 4,10,12),&( 2, 9,11, 5, 8, 7)( 4,12,10),&( 2, 9)( 4,12)( 5, 8)( 7,11),\\
( 3, 6),&( 3, 6)( 7, 9)( 8,11)(10,12),&( 2, 8,11)( 3, 6)( 4,12,10)( 5, 9, 7),\\
( 2, 8)( 3, 6)( 4,12)( 5, 9),&( 2,11)( 3, 6)( 4,10)( 5, 7),&( 2,11, 8)( 3, 6)( 4,10,12)( 5, 7, 9).
\end{array}$
\end{center}

\end{example}

\begin{example}
\label{ex2}
Let $M $ be the group $C_4\times C_2$. This group is recognized by the command ``SmallGroup$(8,2)$'' in GAP and its Cayley table is presented below.

\begin{center}
\begin{tabular}{C|CCCCCCCC}
*&1&2&3&4&5&6&7&8\\
\hline
1&1&2&3&4&5&6&7&8\\
2&2&4&5&6&7&1&8&3\\
3&3&5&1&7&2&8&4&6\\
4&4&6&7&1&8&2&3&5\\
5&5&7&2&8&4&3&6&1\\
6&6&1&8&2&3&4&5&7\\
7&7&8&4&3&6&5&1&2\\
8&8&3&6&5&1&7&2&4\\
\end{tabular}
\end{center}

The Bol loop $L_M$ is nonassociative and has order $32$. Furthermore, since $Aut(M) = D_8$, the dihedral group of order $8$, we have that

\begin{center}
$Aut(L_M) \cong (S_3 \times D_8) \stackrel{\sigma}{\ltimes} C_2^{4}$ and $Half(L_M) \cong C_2 \times ((S_3 \times D_8) \stackrel{\sigma}{\ltimes} C_2^{4})$
\end{center}
by Theorem \ref{t41}. Then $L_M$ has $1536$ half-automorphisms, from which $768$ are proper. Considering $K = \{1,a,b,c\}$ as the Klein group, the permutation 

\begin{center}
$((a,1),(a,4))((a,3),(a,7))((b,1),(b,4))((b,3),(b,7))((c,2),(c,6))((c,5),(c,8))$
\end{center}
is an example of a proper half-automorphism of $L_M$ which does not fix the subgroup $(K,1)$ of $L_M$. This permutation was obtained by using GAP computing with the LOOPS package.
\end{example}

\section*{Acknowledgments}
Some calculations in this work have been made by using the LOOPS package \cite{NV1} for GAP \cite{gap}.



\begin{thebibliography}{99}




\bibitem{G21} G.S. dos Anjos. {\it Half-automorphism group of Chein loops,} to appear in Acta Sci. Math. (Szeged).

\bibitem{FKP06} T. Foguel, M.K. Kinyon, J.D. Phillips. {\it On twisted subgroups and Bol loops of odd order}, Rocky Mountain J. Math., \textbf{36} (1), (2006), 183--212



\bibitem{GG} S. Gagola III, M.L. Merlini Giuliani {\it Half-isomorphisms of Moufang loops of odd order}, Journal of Algebra and Its Applications, \textbf{11}, (2012), 194-199.

\bibitem{GG2} S. Gagola III, M.L. Merlini Giuliani {\it On half-automorphisms of certain Moufang loops with even order}, Journal of Algebra, \textbf{386}, (2013), 131-141.

\bibitem{gap} The GAP Group, GAP - Groups, Algorithms, and Programming, Version 4.10.1, 2019.
\url{http://www.gap-system.org}

\bibitem{GGRS} A. Grishkov, M.L. Merlini Giuliani, M. Rasskazova, L. Sabinina {\it Half-isomorphisms of finite Automorphic Moufang loops}, Communications in Algebra, \textbf{44}, (2016), 4252-4261.

\bibitem{KSV} M. Kinyon, I. Stuhl, P. Vojt\v{e}chovsk\'y {\it Half-Isomorphisms of Moufang Loops}, Journal of Algebra, \textbf{450}, (2016), 152-161.

\bibitem{GA21} M.L. Merlini Giuliani, G.S. dos Anjos,  {\it Half-isomorphisms of automorphic loops}, submitted.

\bibitem{GA19} M.L. Merlini Giuliani, G.S. dos Anjos,  {\it Half-isomorphisms of dihedral automorphic loops}, Communications in Algebra, \textbf{48}, (2020), no. 3, 1150-1162.

\bibitem{GA20} M.L. Merlini Giuliani, G.S. dos Anjos,  {\it Lie automorphic loops under half-automorphisms}, Journal of Algebra and Its Applications, \textbf{19}, (2020), no. 11, 2050221.

\bibitem{GPS} M.L. Merlini Giuliani, P. Plaumann, L. Sabinina  {\it Half-automorphisms of Cayley-Dickson loops}, In: Falcone, G., ed., Lie Groups, Differential Equations and Geometry. Cham: Springer (2017), 109-125.

\bibitem{NV1} G.P. Nagy, P. Vojt\v echovsk\' y. LOOPS: Computing with quasigroups and loops in GAP, version 3.4.0, package for GAP, \url{https://cs.du.edu/~petr/loops/}

\bibitem{P90} H.O. Pflugfelder,  {\it Quasigroups and Loops: Introduction}, Sigma Series in Pure Math. \textbf{7}, Heldermann, 1990.

\bibitem{R66} D.A. Robinson, \textit{Bol loops}, Trans. Amer. Math. Soc.,  \textbf{123}, (1966), 341--354. 

\bibitem{Scott} W.R. Scott {\it Half-homomorphisms of groups}, Proc. Amer. Math. Soc., \textbf{8}, (1957), 1141-1144.
\end{thebibliography}
\end{document}